\documentclass[12pt]{article}

\usepackage{amsmath,amssymb}

\newtheorem{theorem}{Theorem}[section]
\newtheorem{proposition}[theorem]{Proposition}

\newtheorem{corollary}[theorem]{Corollary}
\newtheorem{lemma}[theorem]{Lemma}
\newtheorem{definition}[theorem]{Definition}
\newtheorem{example}[theorem]{Example}
\newtheorem{remark}[theorem]{Remark}

\newenvironment{proof}{\smallskip\par{\sc Proof.}\enspace}%
 {{\unskip\nobreak\hfil\penalty50\hskip2em
          \hbox{}\nobreak\hfil{\rule[-1pt]{5pt}{10pt}}
          \parfillskip=0pt\finalhyphendemerits=0
          \par\medskip}} 

\makeatletter
\def\section{\@startsection {section}{1}{\z@}{3.25ex plus 1ex minus
 .2ex}{1.5ex plus .2ex}{\large\bf}}
\def\subsection{\@startsection{subsection}{2}{\z@}{3.25ex plus 1ex minus
 .2ex}{1.5ex plus .2ex}{\normalsize\bf}}
\@addtoreset{equation}{section} 
\makeatother

\title{ Standardizing densities on Gaussian spaces }

\author{Alberto Lanconelli\footnote{Dipartimento di Matematica, Universit\'a degli Studi di Bari Aldo Moro, Via E. Orabona 4, 70125 Bari - Italia. E-mail: \emph{alberto.lanconelli@uniba.it}}}

\date{\empty}

\begin{document}

\maketitle

\numberwithin{equation}{section}

\bigskip

\begin{abstract}
In the present note we investigate the problem of standardizing random variables taking values on infinite dimensional Gaussian spaces. In particular, we focus on the transformations induced on densities by the selected standardization procedure. We discover that, under certain conditions, the Wick exponentials are the key ingredients for treating this kind of problems.
\end{abstract}



\allowdisplaybreaks

\section{Introduction}

The standardization of random variables is one of the elementary tools in probability theory and mathematical statistics. A random variable is, by definition, \emph{standardized} if it has mean zero and variance one. Given a random variable $X$ with mean $E[X]$ and variance $Var(X)$, one may obtain a standardized version of it via the transformation
\begin{eqnarray}\label{standard}
\tilde{X}:=\frac{X-E[X]}{\sqrt{Var(X)}}.
\end{eqnarray}
One of the crucial features of the above standardizing procedure is that the family of Gaussian distributions is closed under that transformation. The leading role of the prescription (\ref{standard}) is established by the central limit theorem, where the sequence under investigation
\begin{eqnarray*}
\frac{X_1+\cdot\cdot\cdot+X_n-nE[X_1]}{\sqrt{nVar(X_1)}}
\end{eqnarray*}
corresponds to the standardized version, according to (\ref{standard}), of $X_1+\cdot\cdot\cdot+X_n$.  \\
Now consider a random variable $X$ whose law is absolutely continuous with respect to the Lebesgue measure and denote by $f_X$ the corresponding density. Observe that the properties of the Lebesgue measure imply that also the law of $\tilde{X}$ possesses a density whose expression is
\begin{eqnarray*}
y\mapsto f_{\tilde{X}}(y)=\sqrt{Var(X)}f_X(\sqrt{Var(X)}y+E[X]).
\end{eqnarray*}
If we consider densities with respect to different reference measures, then the absolute continuity of the law of $X$ does not, in general, even guarantee the absolute continuity of the law of $\tilde{X}$. In this note we will focus on Gaussian reference measures and investigate standardizing procedures for densities with respect to that measure.\\
In finite dimensional Euclidean spaces it is straightforward to pass from densities with respect to the Lebesgue measure to densities with respect to (non degenerate) Gaussian measures and, from this point of view, the scope of our investigation may look artificial. However, when the reference space is infinite dimensional, for instance the classical Wiener space, then no Lebesgue-type measure is available and the standardization problem is far from being trivial. In fact, according to the Cameron-Martin theorem (see for instance Theorem 2.4.5 in \cite{Bogachev}) the measure obtained by the composition of the Wiener measure with a translation is absolutely continuous with respect to the original Wiener measure if and only if the translation is performed along particular directions. Moreover, Wiener measures with different covariances are mutually orthogonal (see Theorem 5.3 in \cite{Kuo Banach}).\\
The problem of standardizing random variables taking values on infinite dimensional Gaussian spaces is quite natural if we think that solutions to stochastic differential equations driven by a Wiener process are precisely objects of this kind. Consider, for instance, the stochastic differential equation
\begin{eqnarray}\label{SDE}
dX_t=b(X_t)dt+dW_t,\quad t\in [0,T]\quad\quad X_0=0
\end{eqnarray}
where $b:\mathbb{R}\to\mathbb{R}$ is a measurable function and $\{W_t\}_{0\leq t\leq T}$ is a one dimensional standard Wiener process. If the function $b$ satisfies the so-called Novikov's condition (see Corollary 5.13 in \cite{KS}), i.e.
\begin{eqnarray*}
E\left[\exp\left\{\frac{1}{2}\int_0^Tb^2(W_t)dt\right\}\right]<+\infty
\end{eqnarray*}
where $E$ denotes the expectation on the probability space where the Wiener process $\{W_t\}_{0\leq t\leq T}$ is defined, then exploiting the Girsanov's theorem (see Theorem 5.1 in \cite{KS}) we can state that the measure induced by the solution $\{X_t\}_{0\leq t\leq T}$ of (\ref{SDE}) is absolutely continuous with respect to the Wiener measure (the law of the Wiener process on the space of continuous functions) with a density given by
\begin{eqnarray*}
\exp\left\{\int_0^Tb(W_t)dW_t-\frac{1}{2}\int_0^Tb^2(W_t)dt\right\}.
\end{eqnarray*}
In the papers \cite{LS 2016} and \cite{L 2017} the authors investigated the validity of local limit theorems, i.e. central limit theorems for densities, for a class of random variables including the example described above. For this type of results one needs to set the mean and covariance of the sequence to coincide with those of the limiting Gaussian measure. In both papers the limiting measure is the Wiener measure whose mean is zero and the covariance is the identity operator. In \cite{LS 2016} the random variables are assumed to be already standardized while in \cite{L 2017} only the assumption of the first moment is relaxed.\\ 
The aim of this note is to analyze the behaviour of the densities, with respect to the Wiener measaure, of random variables when we try to perturb their first and second moments. We will see that, while perturbations of the first moment are quite simple to perform at a good level of generality, perturbations of the second moment are far more delicate and require strong assumptions on the densities of the random variables. The role of the Wick product, and in particular of the Wick exponentials, will appear to be very natural in our approach to the above mentioned problem. Probabilistic interpretations of the Wick product have been already investigated in the papers \cite{DLS 2011}, \cite{LS 2016} and \cite{L 2017} for Gaussian measures, in \cite{LSportelli2011} and \cite{LS 2013} for the Poisson distribution and in \cite{LSportelli} for the chi-square distribution. \\
The paper is organized as follows:  Section 2 is a quick overview on basic definitions and notations from the analysis on abstract Wiener spaces. This theory provides a unified framework for treating the following leading examples: the Euclidean space $\mathbb{R}^d$ endowed with the standard $d$-dimensional Gaussian measure and the classical Wiener space $C_0([0,1])$ of continuous functions starting at zero endowed with the classical Wiener measure. Section 3 introduces the Wick exponentials, key ingredient of our investigation, and some their properties which are then employed in the perturbations of the first and second moments of random variables taking values on an abstract Wiener space.

\section{Framework}

In this section we recall for the reader's convenience few definitions and notations and collect some useful formulas that we utilize throughout the paper. For more details on the subject we refer the interested reader to one of the books \cite{Bogachev}, \cite{Janson}, \cite{Kuo Banach} and \cite{Nualart}. See also the paper \cite{DPV}.
The triple $(H,W,\mu)$ is called \emph{abstract Wiener space} if $(H,\langle\cdot,\cdot\rangle_H)$ is a separable Hilbert space, which is continuously and densely embedded in the Banach space $(W,|\cdot|_W)$, and $\mu$ is a Gaussian probability measure on the Borel sets of $W$ such that
\begin{eqnarray}\label{Gaussian characteristic}
\int_{W}e^{i\langle w,\varphi\rangle}d\mu(w)=e^{-\frac{1}{2}|\varphi|_H^2},\quad\mbox{ for all }\varphi\in W^*.
\end{eqnarray}
Here $W^*\subset H$ denotes the dual space of $W$, which in turn is dense in $H$, and $\langle\cdot,\cdot\rangle$ stands for the dual pairing between $W$ and $W^*$. We will refer to $H$ as the \emph{Cameron-Martin} space of $W$. Observe that
\begin{eqnarray}\label{H=W}
\langle w,\varphi\rangle=\langle w,\varphi\rangle_H\quad\mbox{ if }\quad w\in H\mbox{ and }\varphi\in W^*.
\end{eqnarray}
We point out that $H=\mathbb{R}^d$ when $W$ is the Euclidean space $\mathbb{R}^d$ while $H=H_0^1([0,1])$ when $W$ is the classical Wiener space $C_0([0,1])$. In the sequel we denote by $\Vert\cdot\Vert_p$ the norm in the space $\mathcal{L}^p(W,\mu)$ for $p\geq 1$.\\
\noindent Recall that by the Wiener-It\^o chaos decomposition theorem any element $f$ in $\mathcal{L}^2(W,\mu)$ has an infinite orthogonal expansion
\begin{eqnarray*}
f=\sum_{k\geq 0}\delta^k(f_k),
\end{eqnarray*}
where $f_k\in H^{\hat{\otimes}k}$, the space of symmetric elements of $H^{\otimes k}$, and $\delta^k(f_k)$ stands for the multiple It\^o integral of $f_k$. Here, $\delta^0(f_0)$ coincide with the mean of $f$ with respect to $\mu$ and $\delta^1(f_1)$ will be denoted in the sequel with the symbols $\delta(f_1)$ or $\langle w,f_1\rangle$.\\
For $f,g\in\mathcal{L}^2(W,\mu)$ with $f=\sum_{k\geq 0}\delta^k(f_k)$ and $g=\sum_{k\geq 0}\delta^k(g_k)$ one has the identity
\begin{eqnarray}\label{orthogonality chaos expansion}
\int_Wf(w)g(w)d\mu(w)=\sum_{k\geq 0}k!\langle f_k,g_k\rangle_{H^{\otimes k}}.
\end{eqnarray}
On the space $\mathcal{L}^2(W,\mu)$ one can define a multiplication between functions through the prescription
\begin{eqnarray}\label{def wick product}
\delta^k(f_k)\diamond\delta^j(f_j):=\delta^{k+j}(f_k\hat{\otimes} f_j),\quad k,j\geq 0
\end{eqnarray}
where $\hat{\otimes}$ denotes the symmetric tensor product. This is named \emph{Wick product} of $\delta^k(f_k)$ and $\delta^j(f_j)$ and it is extended by linearity to finite sums of multiple It\^o integrals. We remark that this product is unbounded on $\mathcal{L}^p(W,\mu)$ for any $p\geq 1$. \\
For additional information on the Wick product (and its role in the theory of stochastic differential equations)  we refer to the book \cite{HOUZ}, the paper \cite{DLS 2013} and the references quoted there.

\section{Standardized densities}

In this section we investigate the problem of standardizing random variables taking values on an abstract Wiener space. More precisely, for a given such random object we will be interested in finding a procedure for setting its mean to zero and its covariance to be the identity operator. In particular, we will focus on how this procedure will modify the density of the original random variable.

\subsection{Wick exponentials}

We begin by introducing the key tool of our investigation.

\begin{definition}\label{def wick exponential}
Let $Z\in\mathcal{L}^2(W,\mu)$ be such that the series
\begin{eqnarray}\label{series}
\sum_{n\geq 0}\frac{Z^{\diamond n}}{n!}\quad\mbox{ where }\quad Z^{\diamond n}:=Z\diamond\cdot\cdot\cdot\diamond Z\mbox{ ($n$-times)}
\end{eqnarray}
converges in $\mathcal{L}^p(W,\mu)$ for some $p\geq 1$. Then, the \emph{Wick exponential} of $Z$ is defined to be
\begin{eqnarray}
\exp^{\diamond}\{Z\}:=\sum_{n\geq 0}\frac{Z^{\diamond n}}{n!}.
\end{eqnarray}
\end{definition}

\noindent Wick exponentials are sometimes denoted in the literature with the symbol $:\exp\{Z\}:$ (see for instance \cite{HKPS} and \cite{Kuo}). It is easy to see, exploiting the basic properties of the Wick product, that
\begin{eqnarray*}
\exp^{\diamond}\{Z+Y\}=\exp^{\diamond}\{Z\}\diamond\exp^{\diamond}\{Y\}
\end{eqnarray*}
and
\begin{eqnarray}\label{functor}
\int_W\exp^{\diamond}\{Z\}d\mu=\exp\left\{\int_WZd\mu\right\}.
\end{eqnarray}
In general, given $Z\in\mathcal{L}^2(W,\mu)$ the series (\ref{series}) converges in the topology of the Kondratiev's distribution spaces (see \cite{HOUZ}) which constitute a family of spaces of stochastic generalized functions containing $\mathcal{L}^p(W,\mu)$ for all $p> 1$. Nevertheless, for some particular choices of $Z$, the series (\ref{series}) possesses the required convergence property.\\
We observe that
\begin{eqnarray*}
\exp^{\diamond}\{Z\}=\exp\{Z\}
\end{eqnarray*}
when  $Z$ is constant. The two most important Wick exponentials are those related to the choices $Z=\delta(h)$ and $Z=\delta(h)^2$, for $h\in H$.

\begin{definition}
Let $h\in H$. Then, we set
\begin{eqnarray}\label{def linear stochastic exponential}
\mathcal{E}(h):=\exp^{\diamond}\{\delta(h)\}
\end{eqnarray}
and
\begin{eqnarray}\label{def quadratic stochastic exponential}
\mathcal{E}_2(h):=\exp^{\diamond}\left\{\frac{\delta(h)^2-|h|_H^2}{2}\right\}.
\end{eqnarray}
\end{definition}

\noindent We remark that the additive constant $-|h|_H^2$ in the exponent of (\ref{def quadratic stochastic exponential}) serves to guarantee the equality
\begin{eqnarray*}
\int_W\mathcal{E}_2(h)d\mu=1
\end{eqnarray*}
in virtue of formula (\ref{functor}). This fact, in connection with the positivity of $\mathcal{E}_2(h)$ proven below, will suggest to interpret $\mathcal{E}_2(h)$ as a probability density with respect to the Wiener measure $\mu$.

\begin{proposition}
For any $h\in H$ the Wick exponential $\mathcal{E}(h)$ belongs to $\mathcal{L}^p(W,\mu)$ for all $p\geq 1$ and it can be represented as
\begin{eqnarray*}
\mathcal{E}(h)=\exp\left\{\delta(h)-\frac{|h|_H^2}{2}\right\}.
\end{eqnarray*}
For $h\in H$ the Wick exponential $\mathcal{E}_2(h)$ belongs to $\mathcal{L}^p(W,\mu)$ if and only if $|h|_H^2<\frac{1}{p-1}$ and it can be represented as
\begin{eqnarray*}
\mathcal{E}_2(h)=\frac{1}{\sqrt{1+|h|_H^2}}\exp\left\{\frac{\delta(h)^2}{2(1+|h|_H^2)}\right\}.
\end{eqnarray*}
\end{proposition}

\begin{proof}
The statement concerning the Wick exponential $\mathcal{E}(h)$ is well known and can be verified directly by means of the generating function of the Hermite polynomials (see for instance Theorem 3.33 in \cite{Janson}). Moreover, observing that $\delta(h)^2-|h|_H^2=\delta(h)^{\diamond 2}$, one can find in Proposition 3.5 of \cite{AOU} the claimed representation for $\mathcal{E}_2(h)$. This in turn can be used to check the condition for the integrability of $\mathcal{E}_2(h)$.
\end{proof}

\subsection{Setting the mean to zero}

We are now going to show the role of the Wick exponential $\mathcal{E}(h)$ in the process of standardizing random variables taking values on Gaussian spaces. To this aim we recall the crucial interplay between the Wick product $\diamond$ and $\mathcal{E}(h)$:
\begin{eqnarray}\label{S-transform}
\int_W (f\diamond g)\cdot\mathcal{E}(h)d\mu=\int_W f\cdot\mathcal{E}(h)d\mu\cdot\int_W g\cdot\mathcal{E}(h)d\mu
\end{eqnarray}
which is valid for all $h\in H$ and $f,g\in\mathcal{L}^2(W,\mu)$ such that $f\diamond g\in\mathcal{L}^p(W,\mu)$ for some $p>1$.

\begin{theorem}\label{set mean zero}
Let $X$ be a random variable taking values on $W$. Assume that the law of $X$ is absolutely continuous with respect to $\mu$ with a density $f$ belonging to $\mathcal{L}^p(W,\mu)$ for some $p> 1$. Then, the expectation $E[X]$ of $X$ belongs to $H$ and coincides with the first kernel in the chaos decomposition of $f$. Moreover, the density of $X-E[X]$ is given by $f\diamond\mathcal{E}(-E[X])$ and belongs to $\mathcal{L}^q(W,\mu)$ for all $q<p$.
\end{theorem}

\begin{remark}\label{existence of mean}
The expectation of $X$ is defined to be the unique element $E[X]\in W$ such that
\begin{eqnarray}\label{def expectation}
E[\langle X,\varphi\rangle]=\langle E[X],\varphi\rangle,\quad\mbox{ for all }\quad\varphi\in W^*.
\end{eqnarray}
Its existence is guaranteed by the finiteness of all the moments of $\langle X,\varphi\rangle$, which in turn is implied by the assumption $f\in\mathcal{L}^p(W,\mu)$. In fact, for any $m\in\mathbb{N}$ and $\varphi\in W^*$ a simple application of the H\"older inequality gives
\begin{eqnarray*}
E[|\langle X,\varphi\rangle|^m]&=&\int_W|\langle w,\varphi\rangle|^md\mu_{X}(w)\\
&=&\int_W|\langle w,\varphi\rangle|^m f(w)d\mu(w)\\
&\leq&\Big(\int_W|\langle w,\varphi\rangle|^{qm} d\mu(w)\Big)^{\frac{1}{q}}\cdot\Vert f\Vert_p\\
&\leq&C(q)|\varphi|_H^{m}\Vert f\Vert_p
\end{eqnarray*}
where $q$ is the conjugate exponent of $p$ and $C(q)$ is a positive dependent on $q$. Moreover, taking $m=2$ in the previous inequalities, we see that if $\{\varphi_j\}_{j\geq 1}$ is a sequence in $W^*$ converging in the norm of $H$ to $h\in H$, then
\begin{eqnarray*}
E[|\langle X,\varphi_j\rangle-\langle X,\varphi_i\rangle|^2]&=&E[|\langle X,\varphi_j-\varphi_i\rangle|^2]\\
&\leq&C(q)|\varphi_j-\varphi_i|_H^2\Vert f\Vert_p.
\end{eqnarray*}
Therefore, $\{\langle X,\varphi_j\rangle\}_{j\geq 1}$ turns out to be a Cauchy sequence in $\mathcal{L}^2(W,\mu)$ and we can define $\langle X,h\rangle$ almost surely as the limit of this sequence.
\end{remark}

\begin{proof}
The existence of the mean $E[X]$ is guaranteed by Remark \ref{existence of mean}; moreover,
\begin{eqnarray*}
E[\langle X,\varphi\rangle]&=&\int_W\langle w,\varphi\rangle f(w)d\mu(w)\\
&=&\langle f_1,\varphi\rangle_H
\end{eqnarray*}
where we utilized equation (\ref{orthogonality chaos expansion}). Comparing with equation (\ref{def expectation}), this shows that $E[X]=f_1\in H$. Moreover, by Gjessing's lemma (see Theorem 2.10.6 in \cite{HOUZ}) we deduce that $f\diamond\mathcal{E}(-E[X])$ is a non negative element of $\mathcal{L}^q(W,\mu)$ for all $q<p$. Now, let $i$ be the imaginary unit and set for $\varphi\in W^*$
\begin{eqnarray*}
\mathcal{E}(i\varphi):=\exp\left\{i\delta(\varphi)+\frac{|\varphi|_H^2}{2}\right\}.
\end{eqnarray*}
Using identity (\ref{S-transform}) (with $\mathcal{E}(h)$ replaced by $\mathcal{E}(i\varphi)$) we can write
\begin{eqnarray*}
\int_We^{i\delta(\varphi)}(f\diamond\mathcal{E}(-E[X]))d\mu&=&e^{-\frac{|\varphi|_H^2}{2}}
\int_W\mathcal{E}(i\varphi)(f\diamond\mathcal{E}(-E[X]))d\mu\\
&=&e^{-\frac{|\varphi|_H^2}{2}}\int_W\mathcal{E}(i\varphi)fd\mu
\int_W\mathcal{E}(i\varphi)\mathcal{E}(-E[X])d\mu\\
&=&e^{-\frac{|\varphi|_H^2}{2}-i\langle E[X],\varphi\rangle_H}\int_W\mathcal{E}(i\varphi)fd\mu\\
&=&e^{-i\langle E[X],\varphi\rangle_H}\int_We^{i\delta(\varphi)}fd\mu\\
&=&e^{-i\langle E[X],\varphi\rangle_H}E\Big[e^{i\langle X,\varphi\rangle}\Big]\\
&=&E\Big[e^{i\langle X-E[X],\varphi\rangle}\Big].
\end{eqnarray*}
In other words, the density of $X-E[X]$ with respect to $\mu$ is $f\diamond\mathcal{E}(-E[X])$. The proof is complete.
\end{proof}

\subsection{Setting the covariance to be the identity}

We now turn our attention to the covariance of the random variable $X$. As it was mentioned in the introduction, the laws of $X$ and $aX$, where $a$ is a real number with $|a|\neq 1$, are in general singular to each other. Therefore, if we want to preserve the absolute continuity with respect to the underlying Wiener measure, we should perturb the covariance through a different procedure. We are interested in setting the covariance of $X$ to coincide with the identity operator, which is the covariance of the Wiener measure. Theorem \ref{set variance one} below suggests a solution to this problem for a particular class of densities. We first need the following auxiliary result.

\begin{lemma}\label{lemma}
If $f$ is a non negative element of $\mathcal{L}^p(W,\mu)$ for some $p>1$, then there exists a constant $C=C(p)$ such that, for all $h\in H$ satisfying $|h|_H<C$, the quantity $f\diamond\mathcal{E}_2(h)$ belongs to $\mathcal{L}^q(W,\mu)$ for all $q<p$. Moreover, $f\diamond\mathcal{E}_2(h)$ is also non negative.
\end{lemma}

\begin{proof}
We first prove the statement about the sign of $f\diamond\mathcal{E}_2(h)$. The Wick exponential $\mathcal{E}_2(h)$ (as well as $\mathcal{E}(h)$) belongs to the class of the so-called strongly positive random variables (see \cite{Nualart Zakai}). In fact, via a simple verification one sees that $\mathcal{E}_2(h)$ can be also represented as
\begin{eqnarray*}
\mathcal{E}_2(h)=\int_{\mathbb{R}}\mathcal{E}(\lambda h)d\mu_1(\lambda).
\end{eqnarray*}
where $\mu_1$ denotes the standard one dimensional Gaussian measure. The last identity, together with Theorem 5.1 in \cite{Nualart Zakai}, implies the above mentioned property. When we Wick-multiply a non negative random variable ($f$ in our case) with a strongly positive element, we obtain a new non negative random variable (see \cite{Nualart Zakai}).\\
We now turn to the claim on the integrability of $f\diamond\mathcal{E}_2(h)$.  First of all, we observe that, by means of the identity (\ref{S-transform}), one can verify that
\begin{eqnarray*}
f\diamond\mathcal{E}_2(h)=\int_{\mathbb{R}}f\diamond\mathcal{E}(\lambda h)d\mu_1(\lambda).
\end{eqnarray*}
Therefore, the desired $\mathcal{L}^q(W,\mu)$-integrability of $f\diamond\mathcal{E}_2(h)$ will be a consequence of the condition
\begin{eqnarray*}
\int_{\mathbb{R}}\Vert f\diamond\mathcal{E}(\lambda h)\Vert_q d\mu_1(\lambda)<+\infty.
\end{eqnarray*}
According to the Gjessing's lemma and H\"older inequality we have
\begin{eqnarray*}
\int_{\mathbb{R}}\Vert f\diamond\mathcal{E}(\lambda h)\Vert_q d\mu_1(\lambda)&=&\int_{\mathbb{R}}\Vert (T_{-\lambda h}f)\cdot\mathcal{E}(\lambda h)\Vert_q d\mu_1(\lambda)\\
&\leq&\int_{\mathbb{R}}\Vert T_{-\lambda h}f\Vert_{q+\varepsilon}\cdot\Vert\mathcal{E}(\lambda h)\Vert_{q_1} d\mu_1(\lambda)\\
&=&\int_{\mathbb{R}}\Vert T_{-\lambda h}f\Vert_{q+\varepsilon}\cdot\exp\left\{\frac{q_1-1}{2}\lambda^2|h|_H^2\right\} d\mu_1(\lambda)\\
&\leq&\Vert f\Vert_{q+2\varepsilon}\int_{\mathbb{R}}\exp\left\{\frac{1}{2\varepsilon}\lambda^2|h|_H^2\right\}  \cdot\exp\left\{\frac{q_1-1}{2}\lambda^2|h|_H^2\right\} d\mu_1(\lambda).
\end{eqnarray*}
Here, $T_{-\lambda h}$ denotes the translation operator (see for instance Chapter 10 in \cite{Kuo}), $\varepsilon$ is an arbitrary positive small number while the real number $q_1$ satisfies $\frac{1}{q+\varepsilon}+\frac{1}{q_1}=\frac{1}{q}$. Moreover, in the last inequality we utilized estimate (14.7) of Theorem 14.1 from \cite{Janson}. The last integral is finite if $|h|^2_H$ is small enough (depending on $\varepsilon$ and $q$). Setting $p=q+2\varepsilon$ we get the desired property.
\end{proof}

\noindent We are ready for the main result of the present subsection. The constant $C$ appearing in the statement of the theorem is the one from the previous lemma.

\begin{theorem}\label{set variance one}
Let $X$ be a random variable taking values on $W$ and assume that the law of $X$ is absolutely continuous with respect to $\mu$ with a density $f$ belonging to $\mathcal{L}^p(W,\mu)$ for some $p>1$. We suppose that $E[X]=0$ and
\begin{eqnarray}\label{hp covariance}
cov(\langle X,\varphi_1\rangle,\langle X,\varphi_2\rangle)=\langle\varphi_1,\varphi_2\rangle_H-\langle g,\varphi_1\rangle_H\langle g,\varphi_2\rangle_H
\end{eqnarray}
for some $g\in H$ with $|g|_H<C\wedge 1$ and all $\varphi_1,\varphi_2\in H$. Then, if $Z$ denotes a one dimensional standard Gaussian random variable independent of $X$, we have that
\begin{eqnarray}\label{new X}
\hat{X}:=X+Zg
\end{eqnarray}
has covariance
\begin{eqnarray}\label{covariance}
cov(\langle\hat{X},\varphi_1\rangle,\langle\hat{X},\varphi_2\rangle)=\langle\varphi_1,\varphi_2\rangle_H.
\end{eqnarray}
Moreover, the density of $\hat{X}$ with respect to $\mu$ is given by
\begin{eqnarray*}
\hat{f}:=f\diamond\mathcal{E}_2(g).
\end{eqnarray*}
\end{theorem}

\begin{remark}
We note that condition $|g|_H<C\wedge 1$ guarantees on one side the $\mathcal{L}^q(W,\mu)$ integrability of the density $f\diamond\mathcal{E}_2(h)$ (see Lemma \ref{lemma}) and on the other the fact that the right hand side of (\ref{hp covariance}) is non negative definite. We also observe that assumption (\ref{hp covariance}) corresponds to the statement that the covariance of $X$ is given by $I-\Pi_g$ where $I$ is the identity operator on $H$ and $\Pi_g$ is the projection on $g$, i.e. $\Pi_g\varphi=\langle\varphi,g\rangle_H g$.
\end{remark}

\begin{proof}
First of all, we observe that
\begin{eqnarray*}
cov(\langle X,\varphi_1\rangle,\langle X,\varphi_2\rangle)&=&E[\langle X,\varphi_1\rangle\cdot\langle X,\varphi_2\rangle]\\
&=&\int_W\langle w,\varphi_1\rangle\cdot\langle w,\varphi_2\rangle f(w)d\mu(w)\\
&=&\int_W\langle w,\varphi_1\rangle\diamond\langle w,\varphi_2\rangle f(w)d\mu(w)+\langle \varphi_1,\varphi_2\rangle_H\\
&=&\int_W\delta^2(\varphi_1\hat{\otimes}\varphi_2)f(w)d\mu(w)+\langle \varphi_1,\varphi_2\rangle_H\\
&=&2\langle f_2,\varphi_1\hat{\otimes}\varphi_2\rangle_{H^{\otimes 2}}+\langle \varphi_1,\varphi_2\rangle_H.
\end{eqnarray*}
Hence, assumption (\ref{hp covariance}) corresponds to
\begin{eqnarray*}
2\langle f_2,\varphi_1\hat{\otimes}\varphi_2\rangle_{H^{\otimes 2}}=-\langle g,\varphi_1\rangle_H\langle g,\varphi_2\rangle_H
\end{eqnarray*}
that means
\begin{eqnarray*}
f_2=-\frac{1}{2}g^{\otimes 2}.
\end{eqnarray*}
Exploiting the assumptions on $Z$, we can write
\begin{eqnarray*}
cov(\langle\hat{X},\varphi_1\rangle,\langle\hat{X},\varphi_2\rangle)&=&cov(\langle X+Zg,\varphi_1\rangle,\langle X+Zg,\varphi_2\rangle)\\
&=&cov(\langle X,\varphi_1\rangle+Z\langle g,\varphi_1\rangle,\langle X,\varphi_2\rangle_h+Z\langle g,\varphi_2\rangle)\\
&=&cov(\langle X,\varphi_1\rangle,\langle X,\varphi_2\rangle)+\langle g,\varphi_1\rangle_H\cdot\langle g,\varphi_2\rangle_H\\
&=&\langle \varphi_1,\varphi_2\rangle_H.
\end{eqnarray*}
Moreover,
\begin{eqnarray*}
E[\exp\{i\langle \hat{X},\varphi\rangle\}]&=&E[\exp\{i\langle X+Zg,\varphi\rangle\}]\\
&=&E[\exp\{i\langle X,\varphi\rangle\}]\cdot E[\exp\{iZ\langle g,\varphi\rangle_H\}]\\
&=&\int_W\exp\{i\langle w,\varphi\rangle\}f(w)d\mu(w)\cdot\exp\left\{-\frac{\langle g,\varphi\rangle_H^2}{2}\right\}\\
&=&e^{-\frac{|\varphi|_H^2}{2}}\int_W\mathcal{E}(i\varphi)f(w)d\mu(w)\cdot
\int_W\mathcal{E}(i\varphi)\mathcal{E}_2(g)d\mu(w)\\
&=&e^{-\frac{|\varphi|_H^2}{2}}\int_W\mathcal{E}(i\varphi)(f\diamond\mathcal{E}_2(g))(w)d\mu(w)\\
&=&\int_W\exp\{i\langle w,\varphi\rangle\}(f\diamond\mathcal{E}_2(g))(w)d\mu(w).
\end{eqnarray*}
This proves that $f\diamond\mathcal{E}_2(g)$ is the density of the law of $\hat{X}$ with respect to $\mu$.
\end{proof}

\noindent Combining Theorem \ref{set variance one} with Theorem \ref{set mean zero} we obtain the following straightforward corollary.

\begin{corollary}
Let $X$ be a random variable taking values on $W$ and assume that the law of $X$ is absolutely continuous with respect to $\mu$ with a density $f$ belonging to $\mathcal{L}^p(W,\mu)$ for some $p>1$. We assume that
\begin{eqnarray*}
cov(\langle X,\varphi_1\rangle,\langle X,\varphi_2\rangle)=\langle\varphi_1,\varphi_2\rangle_H-\langle g,\varphi_1\rangle_H\langle g,\varphi_2\rangle_H
\end{eqnarray*}
for some $g\in H$ and all $\varphi_1,\varphi_2\in H$. Then, if $Z$ denotes a one dimensional standard Gaussian random variable independent of $X$, we have that
\begin{eqnarray}\label{new X}
\hat{X}:=X-E[X]+Zg
\end{eqnarray}
has mean zero and covariance
\begin{eqnarray*}
cov(\langle\hat{X},\varphi_1\rangle,\langle\hat{X},\varphi_2\rangle)=\langle\varphi_1,\varphi_2\rangle_H.
\end{eqnarray*}
Moreover, the density of $\hat{X}$ with respect to $\mu$ is given by
\begin{eqnarray*}
\hat{f}:=f\diamond\mathcal{E}(-E[X])\diamond\mathcal{E}_2(g).
\end{eqnarray*}
\end{corollary}

\noindent In the next example we show a density $f$ fulfilling all the assumptions of Theorem \ref{set variance one}

\begin{example}
Consider the function
\begin{eqnarray*}
f:=1-\frac{1}{2}\delta^2\left(g^{\otimes 2}\right)+\delta^4(g^{\otimes 4})
\end{eqnarray*}
where $g\in H$ with $|g|_H<C$, for some positive constant $C<1$ to be fixed later. The function $f$ possesses the properties to be a density with respect to $\mu$. In fact, exploiting the interplay between Hermite polynomials, multiple Wiener-It\^o integrals and Wick product we can write
\begin{eqnarray*}
\delta^2\left(g^{\otimes 2}\right)=\delta(g)^2-|g|_H^2
\end{eqnarray*}
and
\begin{eqnarray*}
\delta^4(g^{\otimes 4})=\delta(g)^4-6|g|^2_H\delta(g)^2+3|g|^4_H.
\end{eqnarray*}
These give
\begin{eqnarray*}
f=\delta(g)^4-\left(\frac{1}{2}+6|g|^2_H\right)\delta(g)^2+3|g|^4_H+\frac{1}{2}|g|_H^2+1.
\end{eqnarray*}
Since the discriminant of the right hand side above is negative for $|g|_H=0$, by continuity we can find a constant $C$ such that $f\geq 0$ for $|g|_H<C$. Moreover, $\int_Wf(w)d\mu(w)=1$ by the properties of the multiple It\^o integrals. This shows that $f$ is a density function. We also remark that:
\begin{itemize}
\item $f\in\mathcal{L}^p(W,\mu)$ for any $p\geq 1$ since it has a finite Wiener-It\^o chaos expansion; \\
\item if $f$ is the density of $X$, then $E[X]=0$ because the kernel of order one of $f$ is identically zero;\\
\item $f_2=-\frac{1}{2}g^{\otimes 2}$, which is equivalent to condition (\ref{hp covariance}).
\end{itemize}
Therefore, if $|g|_H$ is small enough, then the function $f$ satisfies all the assumption of Theorem \ref{set variance one}.
\end{example}

\end{document}